\documentclass[a4paper,11pt,reqno]{article}
\usepackage[margin=1.2in]{geometry}

\usepackage{color, colortbl}

\usepackage{amsmath}
\usepackage{amsfonts}
\usepackage{amsthm}
\usepackage{amssymb}
\usepackage{tikz}
\usepackage[all]{xy}
\usepackage{tikz-cd}
\usepackage[colorlinks]{hyperref}
\usepackage{enumerate}
\usepackage{mathtools}
\usepackage{setspace}
\usepackage{tocloft}

\usepackage{cancel}

\usepackage{comment}

\definecolor{lightgray}{gray}{0.9}
\definecolor{darkgreen}{rgb}{0,0.5,0}
\definecolor{darkblue}{rgb}{0,0.1,0.5}

\hypersetup{
    colorlinks=true, 
    linkcolor=darkblue, 
    urlcolor=blue, 
    linktoc=all, 
    citecolor=darkgreen
    }

\newcommand{\doi[2]}{\href{http://dx.doi.org/#1}{#2}}
\usepackage{thmtools}
\usepackage{mathtools}

\newtheoremstyle{introTheorems}
  {\topsep}
  {\topsep}
  {\itshape}
  {0pt}
  {\bfseries}
  {}
  { }
  {\thmname{\bf #1}
  \textnormal{\thmnote{#3}.\!}
  }

\newtheorem{theorem}{Theorem}[section]

\newtheorem{corollary}[theorem]{Corollary}
\newtheorem{definition}[theorem]{Definition}
\newtheorem{acknowledgment}[theorem]{Acknowledgment}
\newtheorem{example}[theorem]{Example}
\newtheorem{lemma}[theorem]{Lemma}

\newtheorem{question}[theorem]{Question}

\theoremstyle{introTheorems}
\newtheorem{introTheorem}{Theorem}

\usepackage[abbrev,msc-links, alphabetic, backrefs]{amsrefs}

\BibSpec{arXiv}{%
  +{}{\PrintAuthors}{author}
  +{,}{ \textit}{title}
  +{}{ \parenthesize}{date}
  +{,}{ arXiv }{eprint}
}

\newcommand{\Rep}{{\rm Rep}}

\newcommand{\cM}{\mathcal{M}}

\newcommand{\V}{\mathcal{V}}

\newcommand{\C}{\mathbb{C}}

\newcommand{\Z}{\mathbb{Z}}

\newcommand{\g}{\mathfrak{g}}
\renewcommand{\sl}{\mathfrak{sl}}

\newcommand{\SL}{\mathrm{SL}}

\newcommand{\osp}{\mathfrak{osp}}

\newcommand{\Y}{\mathcal{Y}}

\renewcommand{\d}{\mathrm{d}}

\newcommand{\rV}{\mathrm{V}} 

\title{Coupling a vertex algebra to a large center}
\author{Boris L. Feigin\footnote{Hebrew University Jerusalem and HSE Moscow}, $\;$ Simon D. Lentner\footnote{University of Hamburg \url{simon.lentner@uni-hamburg.de}}
}
\date{}

\begin{document}

\maketitle

\renewcommand{\O}{\mathcal{O}}
\newcommand{\Sym}{\mathrm{Sym}}
\renewcommand{\det}{\mathrm{det}}
\newcommand{\Aut}{\mathrm{Aut}}

\begin{abstract}
Suppose a Lie group $G$ acts on a vertex  algebra $\V$. In this article we construct a vertex algebra $\tilde{V}$, which is an extension of $\V$ by a big central vertex subalgebra identified with the algebra of functionals on the space of regular $\g$-connections $(\d+A)$. 

The category of representations of $\tilde{\V}$ fibres over the set of connections, and the fibres should be viewed as $(\d+A)$-twisted modules of $\V$, generalizing the familiar notion of $g$-twisted modules. In fact, another application of our result is that it proposes an explicit definition of $(\d+A)$-twisted modules of $\V$ in terms of a twisted commutator formula, and we feel that this subject should be pursued further. 

Vertex algebras with big centers appear in practice as critical level or large level limits of vertex algebras. I particular we have in mind limits of the generalized quantum Langlands kernel, in which case $G$ is the Langland dual and $\V$ is conjecturally the Feigin-Tipunin vertex algebra and the extension $\tilde{\V}$ is conjecturally related to the Kac-DeConcini-Procesi quantum group with big center. With the current article, we can give a uniform and independent construction of these limits. 
\end{abstract}

\vspace{7cm}

\newpage

\tableofcontents

\section{Introduction}

\subsection{Background}

A \emph{vertex algebra} $\V$ is, very roughly speaking, a commutative algebra depending analytically on a coordinate $z$ in the punctured disc $\C^\times$. More precisely, it is a graded vector space with a derivation $T$ and a multiplication map
\begin{align*}
    \Y:\; \V\otimes \V &\to \V[[z,z^{-1}]] \\
    a\otimes b &\mapsto \Y(a,z)b, 
\end{align*}
 taking values in Laurent series in a variable $z$ with coefficients in $\V$. An important axiom is \emph{locality}, which means that $\Y$ is commutative in an analytic sense, and in practice this means a \emph{commutator formula} for the coefficients of $[\Y(a,z),\Y(b,w)]$ depending on the singular terms in $\Y(a,z)b$. 
 
 \bigskip
 
 A zeroth example for a vertex algebra is provided by any commutative algebra $\V$ with a derivation $T$ by the following vertex operator:  
 $$\Y(a,z)b=\sum_{n\geq 0} \frac{z^n}{n!}(T^na)b.$$
 Such vertex algebras are trivial in the sense that there are no singular terms and commutativity holds strictly, but they play a distinct role in the present article as building blocks. A first serious example is the vertex algebra $\mathrm{V}^\kappa(\g)$ associated to an affine Lie algebra $\hat{\g}$ at some fixed level $\kappa$. 
Standard mathematical textbooks on vertex algebras include \cites{Kac97,FBZ04}.
 
 \bigskip

Representations of a vertex algebra form in good cases a braided tensor category. This was anticipated in the physics and analysis context in \cites{MS88,FF88, Verl88, Gab94}. It appeared prominently for affine Lie algebras in \cite{KL93}, which states for generic $\kappa$ a braided tensor equivalence to the category of representations of a quantum group. For arbitrary vertex algebras, the construction of the braided tensor category was made rigorous  under suitable finiteness conditions  in \cite{HL94} and \cite{HLZ06}. 

\bigskip

Frequently, we also encounter vertex algebras $\tilde{\V}$  with a big central subalgebra $\O$. In this case typically the category of representations fibres over the variety $\mathrm{Spec}(\O)$ according to the scalar action $\O$. The variety comes with a distinguished point (i.e. $\O$ is an augmented algebra) and the quotient by the augmentation ideal, or \emph{trivial fibre}, $\V=\tilde{\V}/\O^+$ is  again a vertex algebra. The first prominent examples for such behavior are affine Lie algebras at critical level $\kappa\to -h^\vee$, in which case $\O$ is the algebra of functionals on $\g$-opers \cite{FF92}, see \cite{FBZ04} Proposition~16.8.4. More recently, the large level limits $\kappa\to \infty$ of the generalized quantum Langlands kernels are studied, this is also the topic of our preceding article \cite{FL24}, see references therein. In this case, $\O$ is  the algebra  of functionals on  $\g$-connections (or after reduction again $\g$-opers), and the trivial fibre $\tilde{\V}/\O^+$ is conjecturally the Feigin-Tipunin algebra $\mathcal{W}_p(\g),p\in\mathbb{N}$, whose category of representations is equivalent to the category of representations of a version of the small quantum group $u_q(\g)$ at $2p$-to root of unity \cite{Len25}. 

\bigskip

We expect this to match with the following different approach: Suppose that a Lie group $G$ acts on a vertex algebra $\V$, then we expect a larger category of $(d+A)$-twisted modules fibring over the space of $\g$-connections $\d+A$, and which restrict to usual modules over $\V^G$, as we explained in \cite{FL24b}. To compare  this to a more established setting: Let us restrict ourselves to fibres over regular singular connections $\d+A_0/z+\cdots$, then such a connection is up to gauge equivalence parametrized by the monodromy $g=e^{-A_0}$ at $z=0$, then in this setting we expect a $G$-crossed extension of $\Rep(\V)$ consisting of $g$-twisted $\V$-modules for every $g\in G$ \cite{FLM88}, which are precisely the fibres. To be precise, in terms of the group structure, it is more appropriate to expect a $G^*$-crossed extension for $G^*$ the Poisson dual group. In the quantum group case, we expect representations of the Kac-DeConcini-Procesi quantum group, which has a big center $\O(G^*)$ and zero fibre~$u_q(\g)$. We seem to be missing an algebraic theory of $(\d+A)$-crossed braided tensor categories, probably in a higher categorical sense, see Question \ref{question_higher}.

\bigskip

The goal of the present article is to explicitly construct for any vertex algebra $\V$ with action of $G$ a vertex algebra $\tilde{\V}$ with a big central subalgebra $\O=\O(\g[[z]])$, the algebra of functionals on the set of regular $\g$-connections, such that $\tilde{V}/\O^+=\V$. From this we also gain new insights in the yet-to-be-defined theory of $(\d+A)$-twisted modules. 

In particular we recover the examples of  $(\hat{\sl}_2)_1$ and  of symplectic fermions which we treated in \cite{FL24}. In this previous work, we had constructed $\tilde{\V}$ as large level limit of the GKO-construction $(\hat{\sl}_2)_{\kappa'}\otimes(\hat{\sl}_2)_{\kappa''}$ and of  $\hat{\osp}(2|1)$ as  Hamiltonian reduction of the $N=4$ superconformal algebra. These in turn are special explicit cases of the generalized geometric quantum Langlands kernel for $\sl_2$ and for $p=1$ and $p=2$. As modules over $\V^G$, which is in this case the Virasoro algebra, we find the following behavior: 
\begin{itemize}
    \item For regular connections the Wakimoto modules, the same as for the trivial connection, as expected by gauge transformation.
    \item For regular singular connections $\d+A_0/z+\cdots$ with semisimple $A_0$  the familiar shifts in conformal weights and generic semisimplicity.
    \item For regular singular connections $\d+A_0/z+\cdots$ with nilpotent $A_0$ Jordan blocks in $L_0$ and Verma modules of infinite length.
    \item For irregular connections such as $\d+A_1/z^2+\cdots$ nonzero $L_1,L_2$ eigenvalues and Whittaker modules. 
\end{itemize}

\subsection{Overview}

In Section \ref{sec_commutative} we set up two commutative vertex algebras with several additional structures. The goal is to have, in some sense, a geometric commutative incarnation of the general situation $\tilde{V},V,\O$ that we want. The underlying algebraic geometry is surely not surprising or new to the experts. In the subsequent section we will couple this particular vertex algebra to an arbitrary chosen vertex algebra $\V$ to construct a corresponding vertex algebra $\tilde{V}$ with big center. 

\bigskip

In Section \ref{sec_connections} we consider the commutative vertex algebra $\O(\g[[z]])$, which we interpret as the algebra of functionals on the set of regular connections $\d+A$ with $A\in\g[[z]]$. In addition, it has a Poisson vertex algebra structure, essentially the Lie bracket in $\g$. This vertex algebra also arises as large level limit of $\mathrm{V}^\kappa(\g)$ and the Poisson structure comes from the subleading terms in $\kappa^{-1}$, see \cite{FBZ04} Section 16.

\bigskip

In Section \ref{sec_solutions} we consider the commutative vertex algebra $\O(G[[z]])$, which we interpret as the algebra of functionals on the set of gauge transformations $F\in G[[z]]$ acting on connections by
$$\d+A\mapsto F(\d+A)F^{-1}.$$
In Lemma \ref{lm_correspondenceSpaces} we recall that letting $F$ act on the trivial connection $\d$ gives an bijection between $\g[[z]]$ and $G[[z]]_e$, by which we denote the gauge transformations with $F(0)=e$ the identity. In a slightly different interpretation: $F$ in this bijection is the solution matrix of the differential equation $(\d+A)F=0$, and it can be used to transport the constant solutions of $\d f=0$ to solutions of $(\d+A)f=0$ in any $\g$-representation. Then $F$ is the unique solution of the differential equation up to a choice of initial condition $F(0)\in G$, and conversely any $F$ determines uniquely a connection $\d+A$ with $A=-(\d F)F^{-1}$. Again in slightly different language, $G[[z]]$ is a $G$-bundle over $\g[[z]]$ with a canonical section. 

\bigskip

In Corollary \ref{cor_subalgebra} we state this for the algebras of functionals: There is an isomorphism of $\O(\g[[z]])$ to the $G$-invariant subalgebra of $\O(G[[z]])$:
$$\O(\g[[z]])\cong \O(G[[z]])^G \subset  \O(G[[z]]).$$
In Lemma \ref{lm_VOAmatch} we check that this is compatible with the respective vertex algebra structures, and in Lemma \ref{lm_PoissonModule} we show that the Poisson vertex algebra structure on $\O(\g[[z]])$ can be extended to a Poisson vertex module structure of $\O(G[[z]])$ over $\O(\g[[z]])$. 

\bigskip

In Section \ref{sec_sl2} we discuss all of the above explicitly in the example $\SL_2,\sl_2$. In particular $\O(\g[[z]])$ consists of polynomials in matrix coefficients $a_{-1-k}^*,b_{-1-k}^*,c_{-1-k}^*,d_{-1-k}^*$ and $\O(G[[z]])$ consists of polynomials in matrix coefficients $A_{-1-k}^*,B_{-1-k}^*,C_{-1-k}^*,D_{-1-k}^*$. We explicitly give the embedding of the former as $\SL_2$-invariant quadratic expressions of the latter. 

\bigskip

In Section \ref{sec_automorphisms} we briefly discuss the subgroup of gauge transformations $\Aut(\d+A)\subset G[[z]]$ that fix a given connection. The point is that regular connections have a maximal automorphism group $G$, and that this automorphism group later recovers the action of $G$ on the fibre over the trivial connection.

\bigskip 

In Section \ref{sec_Coupling} we give our main construction, using the structures prepared above:

\bigskip

In Section \ref{sec_CouplingDef} we define for an arbitrary vertex algebra $\V$ with locally finite $G$-action the \emph{vertex algebra with big center}
$$\tilde{V}:=\big(\O(G[[z]])\otimes \V\big)^G.$$
It contains the big central subalgebra $\O(G[[z]])^G\cong \O(\g[[z]])$ as well as $\V^G$.

In Section \ref{sec_CouplingFibre} we define a   linear embedding 
$$\delta:\V\hookrightarrow \tilde{V}$$
and prove that this factors to a linear isomorphism of $\V$ to any single fibre $\tilde{V}/(\phi-\phi(A))$ over any connection $\d+A$. For the zero-fibre, $\delta$ is an isomorphism of vertex algebras. However, over a general connection (or in $\tilde{V}$ altogether)  the additional factor in $\O(G[[z]])$ introduces additional regular terms in the operator product expansion, which modify the singular terms in the operator product expansion of $\V$. One of our main results is the explicit description of the modified commutator formula. This proposes a definition of $(\d+A)$-twisted module. Note that by construction the vertex subalgebra of invariants $\V^G$ is not modified.

\begin{introTheorem}[\ref{thm_twistedRep}]
Let $\d+A$ be a fixed regular connection. Consider a vertex module  
$$\Y_\cM:\,\tilde{V}\otimes \cM\to \cM[[z,z^{-1}]]$$
such that any $\phi$ in the the big central subalgebra $\O(G[[z]])^G\cong \O(\g[[z]])$  acts by the scalar $\phi(A)$. Then the pull-back of $\Y_\cM$ by $\delta$ produces a $\V$-action
$$\Y_\cM:\,V\otimes \cM\to \cM[[z,z^{-1}]]$$
fulfilling the following $(\d+A)$-twisted commutator formula
\begin{align*}
[a_{-m-1}^{\d+A},b_{-n-1}^{\d+A}]
&=\sum_k\sum_{l<0}\left(
\big(\sum_{r+s=-l-1}{-m-1\choose r}A_{-k-1}^{[s]}.a\big)_{-l-1}b
\right)_{-1-(m+n-l-k )}^{\d+A}\\
\intertext{where $A_{-k-1}^{[s]}$ are coefficients of a certain differential polynomial of order $s$ in $A(z)$}
\sum_{k\geq 0}A^{[s]}_{-k-1}z^{k}&:=(\phi\mapsto \phi(F(z)\frac{\d^s}{s!} F(z)^{-1}))
=e^{-\int A(z)\,\d z}\frac{\d^s}{s!}e^{\int A(z)\,\d z}.
\end{align*}

\end{introTheorem}
For example, if the operator product expansion of $a,b$ and all its images under $G$ in $\V$ have at most a second order pole ($l=-2$) then the modified commutator formula reads
\begin{align*}
[a_{-m-1}^{\d+A},b_{-n-1}^{\d+A}]
&=\sum_{k\geq 0}\left( \big((-m-1)\delta_{k=0}+A_{-k-1}\big).a\big)_{1}b
\right)_{-1-(m+n-1-k )}^{\d+A}.
\end{align*}
This corresponds to the formula encountered in \cite{CGL20} and our preceding article \cite{FL24} for symplectic fermions using computations in the large level limit of the generalized geometric quantum Langlands kernel. In the present article we see that for higher poles $l<-2$ the formula is considerably more complicated then we expected.  

\bigskip

In Section \ref{sec_goingback} we explain how, vice versa, for any vertex algebra $\tilde{V}$ with a big central subalgebra $\O=\O(\g[[z]])$ and a compatible  Poisson vertex algebra structure on $\O$ and Poisson module structure on $\tilde{V}$, the gauge transformations produce an action of $G$ on the zero fibre by vertex algebra automorphisms, and for our construction we check that this recovers from $\tilde{\V}$ the initially given action of $G$ on $\V$.

\bigskip

In Section \ref{sec_beyond} we make some steps beyond regular connections, which is the actually interesting case: 

\bigskip

In Section \ref{sec_beyondSolutions} we discuss how for a singular connection $\d+A$ the set of all solutions (which are multivalued) gives again a bundle over the set of connections, but there appear now complications with choosing a section.

\bigskip

In Section \ref{sec_beyondTwisted} we revisit the commutative vertex operator of $\O(\g[[z]])$ and $\O(G[[z]])$ in a more geometric interpretation, using re-expansions. It becomes seemingly clear that the functions over singular connections lead to twisted modules of the vertex algebra $\O(G[[z]])$. We also note that the twisted commutator formula in the previous section seems to continue to hold in this case. In particular it reproduces the correct definition of a $g$-twisted module, which corresponds to the regular singular connection $\d+A_0/z$ with $g=e^{-A_0}$, in the version for possibly non-semisimple $g$ in \cite{Bak16}.

\bigskip

In Section \ref{sec_Doublet} we make our construction concrete for $\SL_2,\sl_2$ and for a vertex algebra~$\V$ generated by a pair of generators $x,y$ which form the standard representation (a so-called doublet vertex algebra). We explicitly compute the images $\delta(x),\delta(y)$. Then we show explicitly how the operator product expansion of an example with second order pole is modified. In particular we recover the examples of  $(\hat{\sl}_2)_1$ and  of symplectic fermions which we treated in \cite{FL24}.

\subsection{Questions}

Repeating our main question on the theory side from \cite{FL24b}:
\begin{question}\label{question_higher}
There seems to be no available algebraic notion of $(\d+A)$-crossed categories for $\d+A$ a connection on the formal punctured disc, and a tensor product related to connections on the $3$-punctured sphere. If we restrict ourselves to regular singular connections, and work up to gauge transformations, this should recover the theory of $G$-crossed braided tensor categories.  

Closest to us seems the approach in \cite{Gait15} in the framework of Lurie's derived algebraic geometry and also motivated from Langlands theory. Gaitsgory sets up a notion of a sheaf of categories over the prestack of local systems. One might expect that working up to gauge transformations hides the higher categorical structure present in this approach.    
\end{question}
\begin{question}
 One would like a theory of $(\d+A)$-twisted modules  of a vertex algebra (with the twisted Jacobi identity obtained in this article). For regular singular connections and up to gauge transformations, this should reduce to the familiar concept of $g$-twisted modules and, going from familiar abelian finite groups to  nonabelian algebraic groups, to a $G^*$-crossed category. 
\end{question}

\begin{question}
If the action of $G$ on $\V$ is inner, then the  twisted modules should all be equivalent to the untwisted modules by using $\Delta$-deformation \cites{DLM96, Li97, AM09}. In this case our construction of $\tilde{V}$ can be expressed as kernel of screening operators. However, we find unclear what the consequences of this setup are. 
\end{question}

\begin{question}
For connections with singularity at $z=0$ the situation is more complicated, partly because the initial value $F(0)$ has to be replaced. For regular singular connections one usually has a distinguished basis of solutions with prescribed monodromy. For irregular connections one uses asymptotic behavior $z\to 0$, but this definition depends on sectors around $z=0$, called Stokes sectors, and the inherent ambiguity is related to the Stokes phenomenon, see for example our short review in  \cite{FL24b}. While we had \emph{not} encountered phenomena associated to Stokes data in this past article (everything there only depends on the formal type of connection), in Section \ref{sec_beyond} it seems that we precisely encounter this ambiguity as a choice of basis for our twisted vertex algebra module, 
\end{question}

\begin{acknowledgment}
    We thank T. Creutzig and T. Dimofte for long stimulating discussions about these questions, and S. Möller and C. Raymond for several helpful comments and questions. 
    
    Both authors thank the Humboldt foundation for past support (such as the Feyodor Lynen fellowship and the Humboldt research prize) and for supporting a short stay in Jerusalem, where much work on this article was done. SL also thanks the SFB 1624 (Higher structures, moduli spaces and integrability) for partial support.
\end{acknowledgment}

\section{A commutative vertex algebra}\label{sec_commutative}

We apologize to those more fluent in algebraic geometry for always fixing a coordinate $z$ and a connection $\d$. We will always explicitly treat the case $\SL_2$, $\sl_2$ as example.

\subsection{Functions on the space of connections}\label{sec_connections}

Let $\g$ be a semisimple complex finite-dimensional Lie algebra. We consider the set of  $\g$-connections on the formal punctured disc $\d+\sum_{n\in\Z} A_nz^{-1-n}$ with $A_n\in \g$. If $A_n=0$ for $n\geq 0$ then we call the connection connection \emph{regular} in $z=0$, if $A_n=0$ for $n>0$ then we call the connection \emph{regular singular} in $z=0$, otherwise we call it \emph{irregular} in $z=0$. 

\begin{definition}\label{def_Og}
We denote by $\O(\g[[z]])$ the algebra of functionals on the space of formal series $\g[[z]]$. We identify it with the algebra of functionals on the space of regular connections $\d+A$ with $A\in\g[[z]]$.  
\end{definition}
Explicitly, it is the free algebra in the generators $\phi_n$ for $n<0$ and where $\phi$ runs through a basis of $\g^*$. We can identify $\g^*$ with $\g$ itself via the invariant form $(-,-)$ and denote the functionals by $a^*=(a,-)$ in $\g^*$ and $a^*_n$ in $\O(\g[[z]])$ respectively. 

\begin{lemma}
The algebra $\O(\g[[z]])$ becomes a commutative vertex algebra with the assignment 
$$\Y(\phi_{-1},z)=\sum_{n\geq 0} \phi_{-n-1}z^{n}
=\sum_{n\geq 0} \frac{T^n}{n!}\phi_{-1}z^{n},$$
where the derivation $T$ is defined by $T\phi_{-n-1}=(n+1)\phi_{-n}$.
Moreover, it carries the structure of a Poisson vertex algebra with 
$$\{a_{-1}{_\lambda} b_{-1}\}=[a,b]_{-1}$$
\end{lemma}

Commutative Poisson vertex algebras arises as the semiclassical limit $\kappa\to \infty$ of vertex algebras, which become commutative in the limit and first order term in $\kappa^{-1}$ defines a Poisson structure, see for example \cite{FBZ04} Chapter 16. Explicitly, $\O(\g[[z]])$ is the large level limit  $\kappa\to \infty$ of 
the vertex algebra $\rV^\kappa(\g)$: For the choice of $\Z[\kappa^{-1}]$-generators $a_n/\kappa$ of $\hat{\g}$ the commutator relation is 
$$[a_n/\kappa,a_m/\kappa]=\kappa^{-1} [a_n,b_m]_{n+m}+\kappa^{-1} \delta_{m+n=0}(a,b).$$ 
This shows that the limit $\kappa\to \infty$ is well-defined and commutative. Moreover, the limit carries the following structure of a Poisson vertex algebra
$$\{a_n/\kappa_\lambda b_m/\kappa\}=[a_n,b_m]+\delta_{m+n=0}(a,b)\lambda.$$
If we identify $a_n/\kappa$ in the limit with the functional $a^*_n\in \O(\g[[z]])$, then this coincides with the Poisson structure given in Definition \ref{def_Og}.

\subsection{Functions on the space of solutions}
\label{sec_solutions}

For $G$ an algebraic group we consider the group $G[[z]]$ of $G$-valued functions on the formal disc. For matrix groups $G\subset \mathrm{GL}_n$ we can write explicitly $F=\sum_{n\geq 0}F_n z^n$ with $F_n$ matrices such that $F(z)$ fulfills the polynomial identity defining elements in $G$, as a function in $z$. For general groups we can make sense of $F_{n}$ at least on any $G$-representation $V$, and on the level of functionals as we do in formula \eqref{formula_Fn}. 

We denote by $G[[z]]_{e}$ the subgroup of functions with $F(0)=e$, the unit in $G$. There is obviously a split short exact sequence of groups
\begin{align}\label{formula_exactsequence}
G[[z]]_{e} \hookrightarrow
G[[z]] \stackrel{z=0}{\rightleftarrows} G.
\end{align}
As spaces, we have a $G$-bundle $G[[z]]$ over $G[[z]]_e$ by right-multiplication, with a canonical section. 

\bigskip

We can interpret $F\in G[[z]]$ as gauge transformations acting on connections $\d+A$ with $A\in \g[[z]]$ via 
\begin{align}\label{formula_Gauge}
\d+A\mapsto F(\d+A)F^{-1}=\d+(-(\d F)F^{-1}+FAF^{-1})
\end{align}
In particular, it is known and not difficult to see that any regular connections can be obtained as gauge transformation from the trivial connection $\d$, meaning $A=-(\d F)F^{-1}$. This condition can be written as a differential equation for $F$ 
\begin{align}\label{formula_DGL}
F^{-1}\d F+F^{-1}AF&=0
\end{align}
or more suggestively $\d F+AF=0$. Then $F$ can also be interpreted as the solution matrix of this differential equation $(\d+A)f=0$ for $f\in V[[z]]$ in any $G$-representation $V$.   

\begin{lemma}\label{lm_correspondenceSpaces}
The space $\g[[z]]$ is a torsor over $G[[z]]_e$, or explicitly, we have a bijection $G[[z]]_e\cong \g[[z]]$. Here, $F$ is sent to image $\d+A$ of $\d$ under this gauge transformation, which is $A=-(\d F)F^{-1}$ as above, and conversely any connection $\d+A$ is sent to the unique solution of the differential equation \eqref{formula_DGL}
with initial value $F(0)=e$.
\end{lemma}

We now formulate this observation for rings of functionals:

\begin{corollary}\label{cor_subalgebra}
There is an isomorphism of algebras
$$\O(\g[[z]])\cong \O(G[[z]]_e) \cong \O(G[[z]])^G,$$
where the $G$-invariants are taken with respect to the right regular action. 
\end{corollary}

We now consider a commutative vertex algebra structure on $\O(G[[z]])$ as follows: Any functional $\Phi$ on $G$ can be upgraded to a functional $\Phi_{-1}$ on $G[[z]]$ by setting $\Phi_{-1}(F)=\Phi(F(0))$. Moreover, by expanding $\Phi(F(z))$ we define 
\begin{align}\label{formula_Fn}\Phi_{-n-1}(F)=\frac{1}{n!}\frac{\partial^n}{(\partial z)^n}\Phi(F(z))|_{z=0}
\end{align}
We define the vertex operator
$$\Y(\Phi_{-1},z)=\sum_{n\geq 0} \Phi_{-n-1}z^n$$
Of course this can be made more explicit for matrix groups, such as  $\SL_2$ above. For example, we can see directly that this vertex operator is compatible with the relations obtained as coefficients of $\det(F(z))=0$, in fact it is enough to see it is compatible with the derivation $T$ sending $\phi_{-n-1}$ to $(n+1)\phi_{-n}$.

\begin{lemma}\label{lm_VOAmatch}
    The $G$-invariant  vertex subalgebra of $\O(G[[z]])$ is $\O(\g[[z]])$ with the vertex algebra structure given in the previous section.
\end{lemma}
\begin{proof}
    This is clear because the structure of a commutative vertex algebra is given by its derivation, and this is in both cases the the action of the derivative $\frac{\partial}{\partial z}$ on functionals on $\g[[z]]$ and $G[[z]]$ respectively, and $A=-(\d F) F^{-1}$ is a homogeneous polynomial. 
\end{proof}

\begin{lemma}\label{lm_PoissonModule}
The following Poisson vertex module structure of $\O(\g[[z]])$ on $\O(G[[z]])$ extends the Poisson vertex algebra structure on the vertex subalgebra $\O(\g[[z]])$: 
$$\{a_{-1}^*{_\lambda} \Phi_{-1}^*\}=(F\mapsto -\frac{\partial}{\partial t}\Phi(e^{at}F(0)))|_{t=0}$$
\end{lemma}
\begin{proof}
The Poisson action is an action of the Lie algebra $\g$, and as such compatible with the Poisson structure, which is the bracket of the Lie algebra. The action is clearly by derivations with respect to the algebra structure.

The main issue is to check that the definition reduces on the image of the embedding in Lemma~\ref{lm_correspondenceSpaces} to the Poisson vertex algebra structure on $\O(\g[[z]])$: Namely, for
 $\Phi=b^*_{-1}(-(\d F)F^{-1})=(b,-(\d F(0))F^{-1}(0))$ with fixed $b\in\g$ we get the Poisson vertex algebra structure on $\O(\g[[z]])$
\begin{align*}
\{a_{-1}^*,\Phi_{-1}\}
&=-\frac{\partial}{\partial t}(b,-e^{at}(\d F(0))F^{-1}(0)e^{-at})|_{t=0})\\
&=-(b,[a,-(\d F(0))F^{-1}(0)]) \\
&=([a,b],-(\d F(0))F^{-1}(0)]).
\qedhere
\end{align*}
\end{proof}

\subsection{Example: \texorpdfstring{$\SL_2$, $\sl_2$}{SL2}}\label{sec_sl2} 

We make the constructions in Section \ref{sec_connections} and \ref{sec_solutions} explicit for $\SL_2,\sl_2$: We denote by $A^*,B^*,C^*,D^*$ the functionals on $\SL_2$ that send a matrix $\begin{psmallmatrix} a & b \\c & d\end{psmallmatrix}$ to the matrix elements $a,b,c,d$. The action of $g\in \SL_2$ by right multiplication gives a left action on functionals $(g.\Phi)(-)=\Phi(-g)$, which explicitly reads for example
$$\begin{pmatrix} 1 & 1 \\ 0 & 1\end{pmatrix}.
\begin{pmatrix} A^* & B^* \\ C^* & D^*\end{pmatrix}
=\begin{pmatrix} A^* & A^*+B^* \\ C^* & C^*+D^*\end{pmatrix}$$

We denote by  
$A^*_{-n-1},B^*_{-n-1},C^*_{-n-1},D^*_{-n-1}$ for $n\geq 0$ the functionals on $\SL_2[[z]]$ that send $F=\sum_{n\geq 0} F_{-n-1} z^{n}$ to the respective matrix element of $F_{-n-1}$. Note that $F_{-n-1}$ is typically not in $\SL_2$. The algebra of functionals is hence explicitly given by setting the functional $\det(-)=1$ as a series in $z$: 
$$\O(\SL_2[[z]])
= \C[A^*_{-n-1},B^*_{-n-1},C^*_{-n-1},D^*_{-n-1}]/(\det_{-1}-1,\det_{-2},\ldots),$$
$$\det_{-n-1}:=\sum_{i+j=n} A^*_{-i-1}D^*_{-j-1}-B^*_{-i-1}C^*_{-j-1}
$$

Similarly, we denote the functionals $a^*,b^*,c^*,d^*$ on $\sl_2$, which satisfy the linear relation $a^*+b^*=0$, and $a_{-n-1}^*,b_{-n-1}^*,c_{-n-1}^*,d^*_{-n-1}$ on $\sl_2[[z]]$, which all satisfy the linear relation $a^*_{-n-1}+b^*_{-n-1}=0$, in particular all coefficients are again in $\sl_2$.

\bigskip
We now compute explicitly the embedding of functionals corresponding to the map sending $F$ to $A=-(\d F)F^{-1}$ in Lemma \ref{lm_correspondenceSpaces}. We use that the explicit formula for the inverse of a matrix in $\SL_2$
\footnote{By definition, the inverse in an algebraic group is a polynomial expression. For $\SL_n$ this explicitly clear by Cramers rule. For $\mathrm{GL}_n$ the algebra of functions contains an additional generator for the determinant.}
\begin{align}\label{formula_explicitCorrespondence}
&\sum_{n\geq 0} 
\begin{pmatrix} a^*_{-n-1}(A) & b^*_{-n-1}(A) \\c^*_{-n-1}(A) & d^*_{-n-1}(A) \end{pmatrix}
z^n  \\
&=-\sum_{i,j\geq 0} 
\begin{pmatrix} A^*_{-i-1}(F) & B^*_{-i-1}(F) \\C^*_{-i-1}(F) & D^*_{-i-1}(F) \end{pmatrix}
iz^{i-1}
\begin{pmatrix} D^*_{-j-1}(F) & -B^*_{-j-1}(F) \\\ -C^*_{-j-1}(F) & A^*_{-j-1}(F) \end{pmatrix}
z^j \nonumber
\end{align}
We read this as a relation between elements in the two rings of functionals, thereby drop denoting the dependency on $A,F$. For example the first term is
\begin{align*}
\begin{pmatrix} a^*_{-1} & b^*_{-1} \\c^*_{-1} & d^*_{-1} \end{pmatrix}
&=-
\begin{pmatrix} A^*_{-2} & B^*_{-2} \\C^*_{-2} & D^*_{-2} \end{pmatrix}
\begin{pmatrix} D^*_{-1} & -B^*_{-1} \\\ -C^*_{-1} & A^*_{-1} \end{pmatrix}\\
&=\begin{pmatrix}
-A^*_{-2}D^*_{-1}+B^*_{-2}C^*_{-1} &
A^*_{-2}B^*_{-1}-B^*_{-2}A^*_{-1}\\
-C^*_{-2}D^*_{-1}+D^*_{-2}C^*_{-1} &
C^*_{-2}B^*_{-1}-D^*_{-2}A^*_{-1}
\end{pmatrix}
\end{align*}
By construction, these four expressions (and all polynomials in them) are $\SL_2$-invariant, for example
$$\begin{pmatrix} 1 & 1 \\ 0 & 1\end{pmatrix}.
a^*_{-1} =
    \begin{pmatrix} 1 & 1 \\ 0 & 1\end{pmatrix}
    .(-A^*_{-2}D^*_{-1}+B^*_{-2}C^*_{-1})
    =(-A^*_{-2}(C^*_{-1}+D^*_{-1})+(A^*_{-2}+B^*_{-2})C^*_{-1})
    =a^*_{-1}$$
It is in this perspective a nontrivial assertion that these expressions are algebraically independent, that they generate the algebra of $\SL_2$-invariant functionals, and that together with polynomials in $A_{-1}^*,B_{-1}^*,C_{-1}^*,D_{-1}^*$ they generate the algebra of all functionals on $\SL_2[[z]]$. 

\bigskip

We also check empirically the statement of Lemma \ref{lm_VOAmatch}, that is, the structures of the vertex algebras of $\O(G[[z]]),\O(\g[[z]])$ are compatible with the embedding in Lemma \ref{lm_correspondenceSpaces}:
\begin{align*}
    \Y(a_{-n-1}^*,z)
    &=\sum_{m\geq 0} a_{-m-1}z^{m-n}{m \choose n}\\
    \Y\big(-\sum_{i+j-1=n}i A_{-i-1}^*D_{-j-1}^*,z\big)
    &=-\sum_{i+j-1=n}i\sum_{k,l\geq 0} A_{-k-1}^*z^{k-i}{k \choose i}D_{-l-1}^* z^{l-j}{l \choose j}\\
    &=-\sum_{m\geq 0}
    \big(\sum_{k+l-1=m} k A_{-k-1}^*D_{-l-1}^*\big)z^{m-n}{m \choose n},
\end{align*}
where we use the easy binomial identity
$$\sum_{i+j-1=n} i{k \choose i}{l \choose j}=
\sum_{i+j-1=n} k{k-1 \choose i-1}{l \choose j}=
k{m\choose n},
$$
and similarly for $\Y\big(\sum_{i+j-1=n}i B_{-i-1}^*C_{-j-1}^*,z\big)$, so that finally we have as expected
$$\Y(a_{-n-1}^*,z)
=\Y\big(-\sum_{i+j-1=n}i A_{-i-1}^*D_{-j-1}^*+ iB_{-i-1}^*C_{-j-1}^*,z\big).$$

\subsection{Automorphisms}\label{sec_automorphisms}

We come back to gauge transformations $\d\mapsto F(d+A)F^{-1}$. This facilitates an action of $G[[z]]$ on $\O(\g[[z]])$. It extends to the action of $G[[z]]$ on $\O(G[[z]])$ by right multiplication (this is clear from Lemma \ref{lm_correspondenceSpaces}).

\begin{definition}
For a given connection we define $\Aut(\d+A)$ as the subgroup of $G[[z]]$ fixing $\d+A$, that is
$$-(\d F)F^{-1}+FAF^{-1}=A.$$
\end{definition}
For matrix algebras we may rewrite 
$$\d F = [F,A].$$

\begin{example}
For $A=0$ the condition reduces to $\d F=0$, so any constant element in $G$ is an automorphism
$$\Aut(\d)=G.$$
\end{example}
Since these are constants, they are compatible with the derivation $T$ on $\O(G[[z]])$, hence: 
\begin{corollary}
The group $G$ acts on the fibre of $\O(G[[z]])$ over the trivial connection $\d$ by vertex algebra automorphisms.
\end{corollary}

In fact this is the automorphism group for an arbitrary regular connection, because the differential equation has a unique solution with given initial value $F(0)\in G$.

\begin{example}
For $\g=\sl_2$ and constant matrix $A=\begin{pmatrix}
t & 0 \\0 & t^{-1}\end{pmatrix}$ for fixed $t\in\C$ we have the condition
$$\begin{pmatrix}
    A(z)' & B(z)' \\ C(z)' & D(z)'  
\end{pmatrix}
= \begin{pmatrix}
    0 & -B(z)(t-t^{-1}) \\ C(z)(t-t^{-1}) & 0  
\end{pmatrix}.
$$
The solutions give the elements in $\Aut(\d+A)$ to be 
$$F(z)=\begin{pmatrix}
    A & Be^{-(t-t^{-1})z} \\ Ce^{(t-t^{-1})z} & D  
\end{pmatrix}$$
for any constants $A,B,C,D$, and under the condition $AD-BC=1$ clearly $F(z)\in \SL_2$ for all $z$.
\end{example}

For singular connections the automorphism group is generally smaller, because the solutions $F(z)$ lay outside of $G[[z]]$, as the following example demonstrates:

\begin{example}
For $\g=\sl_2$ and $A=\begin{pmatrix}
t & 0 \\0 & t^{-1}\end{pmatrix}/z$ we have the condition
$$\begin{pmatrix}
    A(z)' & B(z)' \\ C(z)' & D(z)'  
\end{pmatrix}
= \begin{pmatrix}
    0 & -B(z)(t-t^{-1})/z \\ C(z)(t-t^{-1})/z & 0  
\end{pmatrix}.
$$
This is solved by the following multivalued function on $\C\backslash \{0\}$
$$F(z)=\begin{pmatrix}
    A & Bz^{-(t-t^{-1})} \\ Cz^{(t-t^{-1})} & D  
\end{pmatrix}.$$
The only automorphisms contained in $\SL_2[[z]]$ are constant diagonal matrices. 
\end{example}

\section{Coupling}\label{sec_Coupling}
\subsection{Definition}\label{sec_CouplingDef}

Let $\V$  be a vertex algebra with an action of a semisimple finite-dimensional Lie group $G$ by automorphisms, denoted by $g.v$.

\begin{definition}\label{def_construction}
We define the \emph{vertex algebra with big center} as the $G$-invariant vertex subalgebra of the product
$$\tilde{\V}:=\big(\O(G[[z]])\otimes \V\big)^G.$$
\end{definition}

Clearly, $\tilde{\V}$ contains the $G$-invariant vertex subalgebras of the two tensor factors, so with the identification in Corollary \ref{cor_subalgebra} 
$$\tilde{V}\supset \O(\g[[z]])\otimes \V^G.$$
We call the first factor the \emph{big central subalgebra} of $\tilde{\V}$.  

\subsection{Description of the fibres}\label{sec_CouplingFibre}

Let $\V$ be a vertex algebra and $G$ a Lie group acting on $\V$ by automorphisms. We assume from now on that the action is locally finite.

\begin{definition}\label{def_delta}
We consider the linear monomorphism 
$$\delta:\,\V\hookrightarrow  \big(\O(G)\otimes \V\big)^G,$$
$$v\mapsto (g\mapsto g^{-1}.v).$$
\end{definition}
    It lands in the $G$-invariant subspace $(\O(G)\otimes \V)^G$ because 
    $$h.\delta(v)=(g\mapsto h.(gh)^{-1}.v)=(g\mapsto g^{-1}.v).$$

Recall that $\O(G)\hookrightarrow \O(G[[z]])$,  as those functionals depending only on $g=F(0)$. 
\begin{definition}\label{def_deltaLoop}
We consider the linear map 
$$\delta:\,\V\hookrightarrow \big(\O(G[[z]])\otimes \V\big)^G,$$
$$v\mapsto (F\mapsto F(0)^{-1}.v)$$
\end{definition}
For a matrix algebra, this can be written
\begin{align}\label{formula_deltaexplicit}
\delta(v)=\sum_{(\Phi,F)} \Phi_{-1}\otimes F^{-1}.v.
\end{align}
where $(\Phi,F)$ runs through some dual basis of matrix elements of $\O(G),G$. We will this more explicitly for $\sl_2$ in Section \ref{sec_sl2}.

\bigskip

Now, recall the big central subalgebra $\O(G[[z]])^G\cong\O(\g[[z]])$ in $\tilde{V}$. For any
 fixed connection $\d+A$ with $A\in \g[[z]]$ we can consider the fibre over $\d+A$, which is isomorphic as a vector space as follows
$$\tilde{V} / (\phi-\phi(A))$$
As a consequence of the exact sequence \eqref{formula_exactsequence} we get our final result

\begin{theorem}\label{thm_fibres}
For any fixed connection $\d+A$ the linear monomorphism $\delta$ in Definition~\ref{def_deltaLoop} induces an isomorphism of vector spaces
$$\V \stackrel{\sim}{\to} \tilde{V} / (\phi-\phi(A))$$
In particular, for $A=0$ this is an isomorphism of vertex algebras.   
\end{theorem}

\subsection{Regularly twisted modules}\label{sec_CouplingTwisted}

We now want to show that a fibre  $\tilde{\V} / (\phi-\phi(A))$ over an arbitrary (regular) connection $\d+A$ gives rise to $(\d+A)$-twisted $\V$-modules - in a sense yet to be defined - with an explicit $(\d+A)$-twisted commutator formula. 

\begin{theorem}\label{thm_twistedRep}
Let $\d+A$ be a fixed regular connection. Consider a vertex module  
$$\Y_\cM:\,\tilde{\V}\otimes \cM\to \cM[[z,z^{-1}]]$$
such that any $\phi$ in the the big central subalgebra $\O(G[[z]])^G\cong \O(\g[[z]])$  acts by the scalar $\phi(A)$. Then the pull-back of $\Y_\cM$ by $\delta$ produces a $\V$-action
$$\Y_\cM:\,\V\otimes \cM\to \cM[[z,z^{-1}]]$$
fulfilling the following $(\d+A)$-twisted commutator formula
\begin{align*}
[a_{-m-1}^{\d+A},b_{-n-1}^{\d+A}]
&=\sum_k\sum_{l<0}\left(
\big(\sum_{r+s=-l-1}{-m-1\choose r}A_{-k-1}^{[s]}.a\big)_{-l-1}b
\right)_{-1-(m+n-l-k )}^{\d+A}\\
\sum_{k\geq 0}A^{[s]}_{-k-1}z^{k}&:=(\phi\mapsto \phi(F(z)\frac{\d^s}{s!} F(z)^{-1}))
=e^{-\int A(z)\,\d z}\frac{\d^s}{s!}e^{\int A(z)\,\d z},
\end{align*}
where $A_{-k-1}^{[s]}$ are coefficients of certain differential polynomials of order $s$ in $A(z)$.

\end{theorem}
For example $A^{[0]}_{-k-1}=\delta_{k,0}$ and $A^{[1]}_{-k-1}=A_{-k-1}$. If the only term appearing is $l=-1$ (that is, we only have a first order pole between $a,b$ and all their images under the action of $G$), then there is no modification to the commutator formula. If there is a second order pole, or $l=-2$, then the modified commutator formula reads
\begin{align*}
[a_{-m-1}^{\d+A},b_{-n-1}^{\d+A}]
&=\sum_{k\geq 0}\left( \big((-m-1)\delta_{k=0}+A_{-k-1}\big).a\big)_{-1}b
\right)_{-1-(m+n-1-k )}^{\d+A}.
\end{align*}

\begin{proof}[Proof of Theorem \ref{thm_twistedRep}]
We abbreviate all vertex operators by $\Y(v,z)=\sum_n v_{-n-1}z^{n}$ and $\Y_\cM(v,z)=\sum_n v_{-n-1}^\cM z^{n}$. Recall that $\delta(v)=(F\mapsto F(0)^{-1}.v)$ and recall that any functional $\phi$ on $G$ is  promoted to a functional $\phi_{-1}(F)=\phi(F(0))$ with $\Y(\phi_{-1},z)=\phi(F(z))$. Thus, as a function in $F$
\begin{align*}
\delta(v)_{-m-1}
&=(F^{-1}_{m-i}.v)_{-i-1}
\end{align*}
where the linear endomorphisms $F^{-1}_k$ of $\V$ are defined by $F(z)^{-1}.v=\sum_{k\geq 0} z^kF^{-1}_k.v$.
\bigskip

Thus the commutator formula for is modified as follows, again as a function in $F$

\begin{align*}
[\delta(a)_{-m-1}^\cM,\delta(b)_{-n-1}^\cM]
&=\sum_{i+i'=m,j+j'=n}[(F^{-1}_i.a)_{-i'-1}^\cM,(F^{-1}_{j}.b)_{-j'-1}^\cM]\\
&=\sum_{i+i'=m,j+j'=n}\sum_{l<0} {-i'-1\choose -l-1}
\big((F^{-1}_{i}.a)_{-l-1}(F^{-1}_{j}.b)\big)_{-1-(i'+j'-l)}^\cM 
\end{align*}
Since $G$ acts as vertex automorphisms we have $$(F(z)^{-1}.a)_{-l-1}(F(w)^{-1}.b)=F(w)^{-1}.\big((F(w)F(z)^{-1}.a)_{-l-1}b\big)$$
Then the above computation continues as follows, and we resort, substitute $i',j'$ and introduce $k=i+(j-t)$ and use the Vandermonde identity on $(-m-1)+i$
\begin{align*}
&[\delta(a)_{-m-1}^\cM,\delta(b)_{-n-1}^\cM]\\
&=\sum_{i+i'=m,j+j'=n}\sum_{t\geq 0}\sum_{l<0} {-i'-1\choose -l-1}\left(
F^{-1}_t.\big((F_{j-t}F^{-1}_{i}.a)_{-l-1}b\big)\right)_{-1-(i'+j'-l)}^\cM  \\
&=\sum_k\sum_{l<0}
\left(\sum_{t\geq 0}F^{-1}_t
\big(\sum_{r+s=-l-1}{-m-1\choose r}\sum_{i+(j-t)=k}F_{(j-t))}
{i \choose s}F^{-1}_{i}.a\big)_{-l-1}b
\right)_{-1-(m+n-l-k-t )}^\cM \\
\intertext{and using the definition of $\delta$ and  substituting $j-t$ }
&=\sum_k\sum_{l<0}\delta\left(
\big(\sum_{r+s=-l-1}{-m-1\choose r}\sum_{i+j=k}F_{j}
{i \choose s}F^{-1}_{i}.a\big)_{-l-1}b
\right)_{-1-(m+n-l-k )}^\cM
\end{align*}

\end{proof}

\subsection{Vice versa: Recovering the \texorpdfstring{$G$}{G}-action}\label{sec_goingback}

We now want to comment on how to go the other way around. The setup is as follows: Suppose $\tilde{\V}$ is a vertex algebra with a big central vertex subalgebra $\O$, then we can look at the fibration of $\tilde{\V}$ over $\mathrm{Spec}(\O)$ with fibre $\tilde{V}_x=\tilde{\V}/(\phi-\phi(x))$, and in particular $\tilde{V}/\O^+$ is a vertex algebra.

Suppose that $\O$ has the additional structure of a Poisson vertex algebra and this extends on $\tilde{\V}$ to the structure of a Poisson vertex module over $\O$. Then we can look the Poisson action of $\O$ on the fibres and in particular on the zero fibre. For each fibre, a certain subalgebra of $\O$ will preserve the respective fibre. 

\begin{example}
We discuss this for the big center vertex algebra we constructed in this article in Definition \ref{def_construction}
$$\tilde{\V}=\big(\O(G[[z]])\otimes \V\big)^G.$$
with big central subalgebra $\O=\O(\g[[z]])$. We have seen in Lemma \ref{lm_PoissonModule} that the Poisson action on the first factor recovers (if exponentiated) the  multiplication of $G$ on $\O(G[[z]])$, and the Poisson action on the second factor $\V$ is trivial by definition. We have  also seen in Section \ref{sec_automorphisms} that the automorphisms of the zero fibre are precisely the constants.  
\bigskip

We have found in Theorem \ref{thm_fibres} that the zero fibre can be identified with $\V$ using $\delta$. Hence we can show that the left contragradient action of $G$ on $\tilde{V}$ from the Poisson module structure recovers the initial action of $G$ on the zero-fibre $\V$ under this identification:
$$g.\delta(v)=g.(F\mapsto F^{-1}(0).v)
=(F\mapsto (g^{-1}F)^{-1}(0).v)
=(F\mapsto F^{-1}(0).(g.v))
=\delta(g.v)$$
\end{example}

\newcommand{\Fnorm}{F^{\mathrm{norm}}}
\newcommand{\CP}{\mathbb{CP}}

\section{Beyond regular connections}\label{sec_beyond}

\subsection{Space of solutions}\label{sec_beyondSolutions}

The bulk of the present article treats regular connections. The case of singular connections is much more interesting and includes the familiar notion of $g$-twisted vertex algebra modules. But there are several difficulties, so we do not aim to develop a full theory at this point, but we want to discuss it and give some indications how to proceed and what to expect and how concrete results can  be obtained and reproduced.

\bigskip

We continue to study $G[[z]]$ as $G$-bundle over $\g[[z]]$. Our interpretation was the space of solutions $F\in G[[z]]$ of the differential equation $(\d+A)F=0$ fibring over the space of connections $\d+A$ with $A\in\g[[z]]$, parametrized by the initial value $F(0)\in G$. If $A(z)$ has now singularities in $z=0$, then the space of solutions is not so easily parametrized. We have discussed this in more depth in \cite{FL24b}.

\bigskip

Let us first consider the \emph{regular singular case} $A=A_0/z+\cdots$. Then it is known that there is a gauge transformation $F\in G[[z]]$ to the normal form 
$$\big(\d+A_0/z\big)\Fnorm(z)=0,$$ 
which is clearly solved by 
$$\Fnorm(z)=e^{-A_0 \log(z)}=z^{-A_0}.$$
\begin{example}\label{exm_Fnorm}
For $\g=\mathfrak{sl}_2$ the solution matrix $\Fnorm$  depends on $A_0$ being diagonalizable or nilpotent
\begin{align*}
\left(\d+\begin{pmatrix} \lambda & 0 \\0 & -\lambda\end{pmatrix}\right)
\Fnorm(z)&=0
&
\Fnorm(z)
&=\begin{pmatrix} z^{-\lambda} & 0 \\0 & z^\lambda\end{pmatrix},\\
%
\left(\d+\begin{pmatrix} 0 & 1 \\0 & 0\end{pmatrix}\right)
\Fnorm(z)&=0
&
\Fnorm(z)
&=\begin{pmatrix}  1 & -\log(z) \\0 & 1\end{pmatrix}.\\
\end{align*}
Solutions of a differential equation involving the second column of the nilpotent case are in literature referred to as \emph{solution of second kind}.
\end{example}
Hence in the regular singular case, we may replace $G[[z]]$ by  the $G[[z]]$-module generated by $\Fnorm$ fibring over the space of connections $A_0/z+\g[[z]]$.

In the irregular case, this becomes even more severe, because there are different preferred bases of solutions depending on the sector by the Stokes phenomenon, as we have reviewed in  \cite{FL24b}. 

\subsection{Singularly twisted modules}\label{sec_beyondTwisted}

Let us revisit the commutative vertex operator for $\O(\g[[t]])$ and $\O(G[[t]])$ in Section \ref{sec_commutative} and more generally $O(X[[t]])$ for any set $X$ from a more geometric perspective. Note that we now switch the notation for the local coordinate from $z$ to $t$. For simplicity we assume $X$ a vector space, by assuming $G$ a matrix group embedded into the vector space of all matrices, but also in the general case the coefficients in $G[[z]]$ can be given meaning, as discussed in Section \ref{sec_solutions}.

\bigskip

Let $\phi$ be a functional on the set of $X$-valued functions on $\CP^1$. If we imagine $\phi$ localized at $z$, then we can write explicitly 
$\phi\in \O^{@z}(X[[t]])=\C[f_{-1}^{@z},f_{-2}^{@z},\ldots]$
in terms of the coefficients $f_n^{@z}$ of an expansion of such a function around $z$, as we did for $z=0$
$$f(t)=\sum_{n\geq 0} f_{-n-1}^{@z} (t-z)^n.$$
Now, if the function $f$ was previously expanded around some $z_1$, then we can re-expand it around $z_2$. In this way we get a pullback map for functionals  
$$\O^{@z_2}(X[[t]])\to \O^{@z_1}(X[[t]]).$$
We make this explicit for $z_1=0$ and $z_2=z$ and with $m=n+k$:
\begin{align*}
f(t)&=\sum_{m\geq 0} f_{-m-1}^{@0} t^m\\
&=\sum_{n,k\geq 0}f_{-n-k-1}^{@0} {n\choose k}z^k(t-z)^n\\
&=\sum_{n\geq 0}f_{-n-1}^{@z} (t-z)^n,\\
f_{-n-1}^{@z}&:=\sum_{k\geq 0}f_{-n-k-1}^{@0}  {n+k\choose n}z^k.\\
\intertext{Of course this coincides with applying directly the Taylor formula to the power series}
f_{-n-1}^{@z}&=\frac{1}{n!}\frac{\partial^n}{\partial t^n}f(t)|_{t=z}.
\end{align*}
 Now left-multiplication with a pulled-back functional reproduces precisely the vertex operator from Section \ref{sec_commutative}:
\begin{align*}
\Y(a_{-n-1},z)
&=\Y(\frac{1}{n!}T^na_{-1},z)
=\frac{1}{n!}\frac{\partial^n}{\partial z^n}
\Y(a_{-1},z)\\
&=\frac{1}{n!}\frac{\partial^n}{\partial z^n}
\sum_{m\geq 0} a_{-m-1} z^m
=
\sum_{k\geq 0}a_{-n-k-1}{n+k\choose n}  z^k 
\end{align*}
 
This suggests to consider also in the singular case  functions of the form $G[[z]]\Fnorm$, which are singular at $z=0$ but regular around generic $z$, as twisted modules over the vertex algebra $\O(G[[z]])$. Note that it is clear from the definition that monodromy of this vertex operator is related to the action of $G$ in the normal form case and some action of $G[[z]]$ in the general case. 

\begin{example}
For $G=\SL_2$ and $\Fnorm(z)
=\begin{psmallmatrix} z^{-\lambda} & 0 \\0 & z^\lambda\end{psmallmatrix}$ 
we get for $F(z)
=\begin{psmallmatrix} A(z) & B(z) \\ C(z) & D(z)\end{psmallmatrix}$

\begin{align*}
\Y(A^*_{-n-1},z)
&=\frac{1}{n!}\frac{\partial^n}{\partial t^n}
(F\Fnorm\mapsto A^*(F(t)\Fnorm(t)))
|_{t=z}\\
&=\frac{1}{n!}\frac{\partial^n}{\partial t^n}
 (F\Fnorm\mapsto A^*\left(
\begin{pmatrix} A(t) & B(t) \\ C(t) & D(t)\end{pmatrix}
\begin{pmatrix} t^{-\lambda} & 0 \\0 & t^\lambda\end{pmatrix}
\right)
|_{t=z}\\
&=\frac{1}{n!}\frac{\partial^n}{\partial t^n}
 (F\Fnorm\mapsto \sum_{m\geq 0}A_{-m-1}t^{m}t^{-\lambda})
|_{t=z}\\
&=
\sum_{m\geq 0} A^*_{-n-k-1} {n+k-\lambda \choose n}z^{n-\lambda},\\
\intertext{and similarly}
\Y(B^*_{-n-1},z)
&=\sum_{m\geq 0} B^*_{-n-k-1} {n+k+\lambda \choose n}z^{n+\lambda}\\
\Y(C^*_{-n-1},z)
&=\sum_{m\geq 0} C^*_{-n-k-1} {n+k-\lambda \choose n}z^{n-\lambda}\\
\Y(D^*_{-n-1},z)
&=\sum_{m\geq 0} D^*_{-n-k-1} {n+k+\lambda \choose n}z^{n+\lambda}.
\end{align*}
\end{example}

It is nice to observe that the $(d+A)$-twisted commutator formulas from Theorem~ \ref{thm_twistedRep} simply continue to hold in this case, and in particular we recover the notion of a $g$-twisted module and we recover the concrete $g$-twisted modules of the symplectic fermions in the next section. 

\begin{example}
Consider the regular singular case with semisimple monodromy in Example~\ref{exm_Fnorm}. We discuss the twisted commutator formula in Theorem \ref{thm_twistedRep}, 
where in this case we compute for $F=\Fnorm=\begin{psmallmatrix} z^{-\lambda} & 0 \\0 & z^\lambda\end{psmallmatrix}$
\begin{align*}
\sum_{k\geq 0}A^{[s]}_{-k-1}z^{k}
&:=(\phi\mapsto \phi(F(z)\frac{\d^s}{s!} F(z)^{-1}))\\
&=\big(\phi\mapsto \phi\big(
\begin{psmallmatrix} {\lambda\choose s}z^{-s} & 0 \\0 & {-\lambda\choose s}z^{-s}\end{psmallmatrix}
\big)\big)
={\lambda a^*_{0}\choose s}
\end{align*}
where in the last equality we plug $\lambda a^*_{-1}$ the polynomial ${X \choose s}$, and this is the only contribution because in the regular singular case $a^*_{k}=0$ for $k>0$. Then the commutator formula simplifies by using the Vandermonde identity backwards
\begin{align*}
[a_{-m-1}^{\d+A},b_{-n-1}^{\d+A}]
&=\sum_k\sum_{l<0}\left(
\big({-m-1+\lambda a_{0}^*\choose -l-1}a\big)_{-l-1}b
\right)_{-1-(m+n-l)}^{\d+A}
\end{align*}
This is precisely the formula in \cite{Bak16} for the $g$-twisted modules of a possibly non-semisimple $g$. A similar match is obtained in the nilpotent case. 
\end{example}

\section{Example: Doublet vertex algebras}\label{sec_Doublet}

We now spell out the construction for $\SL_2,\sl_2$ and a vertex algebra $\V$ generated by a doublet (or standard representation) $\C^2$ with standard basis $x_{-1},y_{-1}$ under this action. the symplectic fermions we have treated explicitly in \cite{FL24}. A similar explicit computation can be done for triplet algebras using the $3$-dimensional representation. 

\bigskip

As discussed in Section \ref{sec_sl2}, the commutative vertex algebra $\O(SL_2)$ consists of polynomials in $A_{-k-1}^*,B_{-k-1}^*,C_{-k-1}^*,D_{-k-1}^*$. We first spell out  Definition \ref{def_construction}
$$\tilde{\V}:=\big(\O(\SL_2[[z]])\otimes \V\big)^{\SL_2},$$
and the linear embedding $\delta$ in Definition \ref{def_delta}
$$\delta:\,\V\to \tilde{V}$$
$$\delta:\,v\mapsto (F\mapsto F^{-1}(0).v),$$
which gives a linear bijection on fibres by Theorem \ref{thm_fibres}. The action of $\SL_2$ on $x_{-1},y_{-1}$ is the standard representation. Then concretely 
\begin{align*}
\delta(x_{-1})
&=\left(\begin{pmatrix}
A(z) & B(z) \\ C(z) & D(z) 
\end{pmatrix}
\mapsto 
\begin{pmatrix}
D_{-1} & -B_{-1} \\ -C_{-1} & A_{-1} 
\end{pmatrix}
.x_{-1}
= D_{-1}x_{-1}-C_{-1}y_{-1} 
\right).
\end{align*}

\begin{corollary}
The image of the doublet generators $x_{-1},y_{-1}$ under $\delta$ are explicitly 
\begin{align*}
\delta(x_{-1})
&=\hphantom{-}D_{-1}^*x_{-1}-C_{-1}^*y_{-1}\\
\delta(y_{-1})
&=-B_{-1}^*x_{-1}+A_{-1}^*y_{-1}
\end{align*}
\end{corollary}
The operator product expansion between $\delta(x)(z)$ and $\delta(y)(z)$ is modified by the non-singular operator product expansions in $\O(\SL_2[[z]])$. For example
\begin{align*}
D^*(z)C^*(w)
&=\big(D^*_{-1}+D^*_{-2}z+\cdots\big)
\big(C^*_{-1}+C^*_{-2}w+\cdots\big)\\
&\vphantom{X^{X^{X^X}}}=(z-w)^0 \big(D^*_{-1}C^*_{-1}+w(D^*_{-1}C^*_{-2}+D^*_{-2}C^*_{-1})+w^2(2D^*_{-3}C^*_{-1}+D^*_{-2}C^*_{-2}+\cdots\big)\cdots\\
&+(z-w)^1 \big(D^*_{-2}C^*_{-1}+wD^*_{-2}C^*_{-2}+\cdots\big)\\
&+(z-w)^2\cdots
\end{align*}
As an example, we now assume that $\V$ is 
the vertex algebra of symplectic fermions with operator product expansion 
$$x(z)\,y(w)\sim (z-w)^{-2}1$$
and $x(z)\,x(w)$ and $y(z)\,y(w)$ zero. 
We can then explicitly compute the modified operator product expansion of $\delta(x),\delta(y)$
 and express it in terms of $\O(\g[[z]])$ by formula \eqref{formula_explicitCorrespondence}. 

\begin{align*}
\delta(x)(z)\,\delta(y)(w)
&=\big(D^*(z)x(z)-C^*(z)y(z)\big)
\big(-B^*(w)x(w)+A^*(w)y(w)\big)\\
&=(z-w)^{-2} \left(
(z-w)^0\big(
D^*_{-1}A^*_{-1}+w(D^*_{-1}A^*_{-2}+D^*_{-2}A^*_{-1})+\cdots\big)\right.\\
&\hphantom{=(z-w)^{-2}}\left.+(z-w)^1 \big(D^*_{-2}A^*_{-1}+wD^*_{-2}A^*_{-2}+\cdots\big)+\cdots\right) \\
&\,-(z-w)^{-2} \left(
(z-w)^0\big(
C^*_{-1}B^*_{-1}+w(C^*_{-1}B^*_{-2}+C^*_{-2}B^*_{-1})-\cdots\big)\right.\\
&\hphantom{=(z-w)^{-2}}\left.+(z-w)^1 \big(C^*_{-2}B^*_{-1}+wC^*_{-2}B^*_{-2}-\cdots\big)+\cdots\right) \\
&=(z-w)^{-2}\det(w)-
(z-w)^{-1} d^*(w)
\\
&\\
\delta(x)(z)\,\delta(x)(w)
&=\big(D^*(z)x(z)-C^*(z)y(z)\big)
\big(D^*(w)x(w)-C^*(w)y(w)\big)\\
&=(z-w)^{-2} \left(
(z-w)^0\big(
-D^*_{-1}C^*_{-1}-w(D^*_{-1}C^*_{-2}+D^*_{-2}C^*_{-1})-\cdots\big)\right.\\
&\hphantom{=(z-w)^{-2}}\left.+(z-w)^1 \big(-D^*_{-2}C^*_{-1}-wD^*_{-2}C^*_{-2}-\cdots\big)+\cdots\right) \\
&-(z-w)^{-2} \left(
(z-w)^0\big(
-C^*_{-1}D^*_{-1}-w(C^*_{-1}D^*_{-2}+C^*_{-2}D^*_{-1})-\cdots\big)\right.\\
&\hphantom{=(z-w)^{-2}}\left.+(z-w)^1 \big(-C^*_{-2}D^*_{-1}-wC^*_{-2}D^*_{-2}-\cdots\big)+\cdots\right) \\
&=-(z-w)^{-1} c_{-1}^*\\
&\\
\delta(y)(z)\,\delta(y)(w)
&=\big(-B^*(z)x(z)+A^*(z)y(z)\big)
\big(-B^*(w)x(w)+A^*(w)y(w)\big)\\
&=(z-w)^{-2} \left(
(z-w)^0\big(
-B^*_{-1}A^*_{-1}-w(B^*_{-1}A^*_{-2}+B^*_{-2}A^*_{-1})-\cdots\big)\right.\\
&\hphantom{=(z-w)^{-2}}\left.+(z-w)^1 \big(-B^*_{-2}A^*_{-1}-wB^*_{-2}A^*_{-2}-\cdots\big)+\cdots\right) \\
&-(z-w)^{-2} \left(
(z-w)^0\big(
-A^*_{-1}B^*_{-1}-w(A^*_{-1}B^*_{-2}+A^*_{-2}B^*_{-1})-\cdots\big)\right.\\
&\hphantom{=(z-w)^{-2}}\left.+(z-w)^1 \big(-A^*_{-2}B^*_{-1}-wA^*_{-2}B^*_{-2}-\cdots\big)+\cdots\right) \\
&=(z-w)^{-1} b_{-1}^*
\end{align*}
This reproduces the deformed operator products expansions which we read off the large level limit of the formulas \cite{CGL20} p. 16, and it reproduces the formulas in our previous article \cite{FL24}. 

Alternatively, we can compute the modification of the commutator formula in Theorem~\ref{thm_twistedRep} for $l=-2$
\begin{align*}
[x_{-m-1}^{\d+A},y_{-n-1}^{\d+A}]
&=\sum_{k\geq 0}\left( \big((-m-1)\delta_{k=0}+A_{-k-1}\big).x\big)_{1}y
\right)_{-1-(m+n-1-k )}^{\d+A},
\end{align*}
as we have already done in \cites{FL24,FL24b}.

\end{document}